\documentclass{amsart}
\usepackage{amssymb}
\usepackage{amscd}
\usepackage[latin1]{inputenc}
\usepackage{amsmath}
\usepackage{amsfonts}
\usepackage{latexsym}
\usepackage{verbatim}
\usepackage{graphicx}
\usepackage{mathrsfs}
\usepackage{epsfig}
\usepackage{color}
\usepackage{faktor}
\newtheorem{theorem}{Theorem}[section]
\newtheorem{corollary}[theorem]{Corollary}
\newtheorem{lemma}[theorem]{Lemma}

\newtheorem{proposition}[theorem]{Proposition}

\newtheorem*{Gtheorem}{Theorem \ref{Gtheorem}}
\newtheorem*{Htheorem}{Theorem \ref{hardcase}}
\newtheorem*{Mcorollary}{Corollary \ref{maincor}}
\newtheorem*{Lcorollary}{Corollary \ref{limit2}}

\theoremstyle{definition}
\newtheorem{definition}[theorem]{Definition}
\newtheorem{example}[theorem]{Example}
\newtheorem{remark}[theorem]{Remark}

\begin{document}
\title[Extender sets and measures of maximal entropy for subshifts]
{Extender sets and measures of maximal entropy for
subshifts}
\date{}
\author{Felipe García-Ramos}
\address{Felipe García-Ramos\\
CONACyT \& Physics Institute of the Universidad Aut\'{o}noma de San Luis Potos\'{\i}\\
Av. Manuel Nava \#6, Zona Universitaria, C.P. 78290 \\
San Luis Potosí, S.L.P.\\
Mexico}
\email{fgramos@conacyt.mx}
\author{Ronnie Pavlov}
\address{Ronnie Pavlov\\
Department of Mathematics\\
University of Denver \\
2390 S. York St. \\
Denver, CO 80208 \\
USA}
\email{rpavlov@du.edu}
\thanks{The second author gratefully acknowledges the support of NSF grant DMS-1500685.}

\begin{abstract}
For countable amenable finitely generated torsion-free $\mathbb{G}$, we prove inequalities relating $\mu(v)$ and $\mu(w)$ for
any measure of maximal entropy $\mu$ on a $G$-subshift and any
words $v, w$ where the extender set of $v$
is contained in the extender set of $w$. Our main results are two generalizations
of the main result of \cite{M}; the first applies to all such $v,w$ when $\mathbb{G} = \mathbb{Z}$, 
and the second to $v,w$ with the same shape for any $\mathbb{G}$.
As a consequence of our results we give new and simpler proofs of several facts about 
synchronizing subshifts (including the main result from \cite{Th}) and we answer a question of Climenhaga.
\end{abstract}


\maketitle

\section{Introduction}

\label{intro}

In this paper, we prove several results about measures of maximal entropy on symbolic dynamical systems (subshifts). Measures of maximal entropy are natural measures, defined via the classical Kolmogorov-Sinai entropy, which also connect to problems in statistical physics, such as existence of phase transitions. 

Our dynamical systems are subshifts, which consist of a compact $X \subseteq \mathcal{A}^{\mathbb{G}}$ (for some finite alphabet $\mathcal{A}$ and a countable amenable finitely generated torsion-free group $\mathbb{G}$) and dynamics given by the $\mathbb{G}$-action of translation/shift maps $\{\sigma_{g}\}_{g \in\mathbb{G}}$ (under which $X$ must be invariant). Subshifts are useful both as discrete models for the behavior of dynamical systems on more general spaces, and as an interesting class of dynamical systems in their own right, with applications in physics and information theory. 

Our main results show that when a word $v$ (i.e. an element of $\mathcal{A}^F$ for some finite $F \subset \mathbb{G}$) is replaceable by another word $w$ in $X$ (meaning that $\forall x \in X$, when any occurrence of $v$ is replaced by $w$, the resulting point is still in $X$), there is a simple inequality relating $\mu(v)$ and $\mu(w)$ for every measure of maximal entropy $\mu$. (As usual, the measure of a finite word is understood to mean the measure of its cylinder set; see Section~\ref{defs} for details.) A formal statement of our hypothesis uses extender sets (\cite{KM}, \cite{OP}); the condition ``$v$ is replaceable by $w$'' is equivalent to the containment $E_{X}(v) \subseteq E_{X}(w)$, where $E_X(u)$ denotes the extender set of a word $u$.
 
 

For $\mathbb{Z}$-subshifts specifically, it is possible to talk about
replacing $v$ by $w$ (and thereby the containment $E_X(v) \subseteq E_X(w)$)
even if their lengths $|v|$ and $|w|$ are different, and our first results
treat this case. 

\begin{Htheorem}
	Let $X$ be a $\mathbb{Z}$-subshift with positive topological entropy, $\mu$ a measure of maximal entropy of $X$%
	,  and $w,v\in L(X)$. If $E_{X}(v)\subseteq E_{X}(w)$, then 
	\begin{equation*}
	\mu(v)\leq\mu(w)e^{h_{top}(X)(|w|-|v|)}. 
	\end{equation*}
\end{Htheorem}

\begin{Mcorollary}
	Let $X$ be a $\mathbb{Z}$-subshift with positive topological entropy, $\mu$ a measure of maximal entropy of $X$%
	, and $w,v\in L(X).$ If $E_{X}(v)=E_{X}(w)$, then for every measure of
	maximal entropy of $X$, 
	\begin{equation*}
	\mu(v)=\mu(w)e^{h_{top}(X)(|w|-|v|)}. 
	\end{equation*}
\end{Mcorollary}

In the class of synchronized subshifts (see Section 3.1 for the definition), $E_{X}(v)=E_{X}(w)$ holds for many
pairs of words of different lengths, in which case Corollary~\ref{maincor}
gives significant restrictions on the measures of maximal entropy. In Section~\ref{apps}, we use
Corollary~\ref{maincor} to obtain results about synchronized subshifts.
These applications include a new proof of uniqueness of measures of maximal entropy under the hypothesis
of entropy minimality (see Theorem~\ref{synchunique}), which was previously shown in \cite{Th} via the much more difficult machinery of countable-state Markov shifts, and the following result which verifies a conjecture of Climenhaga (\cite{Cl}). (Here, $X_S$ represents a so-called $S$-gap subshift; see Definition~\ref{Sgap}.)

\begin{Lcorollary}Let $S\subseteq\mathbb{N}$ satisfy $\gcd(S+1)=1$, let
	$\mu$ be the unique MME on $X_{S}$, and let $\lambda = e^{h_{top}(X_{S})}$. Then 
	$\displaystyle\lim_{n\rightarrow\infty}\frac{\left\vert L_{n}(X_{S})\right\vert
	}{\lambda^n}$ exists and is equal to
	$\displaystyle \frac{\mu(1)\lambda}{(\lambda-1)^{2}}$ when $S$ is infinite and
	$\displaystyle \frac{\mu(1)\lambda (1 - \lambda^{-(\max S) - 1})^2}{(\lambda-1)^{2}}$ when $S$ is finite. 
\end{Lcorollary}

\bigskip In fact, we prove that this limit exists for all
synchronized subshifts where the unique measure of maximal entropy is mixing.

Our second main result applies to countable amenable finitely generated torsion-free $\mathbb{G}$, but only to $v,w$ which have the same shape.
This is unavoidable in a sense, since in general, for $F \neq F'$, there will be no natural way to compare
the configurations with shapes $F^c$ and $F'^c$ in extender sets of words $v \in A^F$ and $w \in A^{F'}$ 
respectively. 



\begin{Gtheorem}
	Let $X$ be a $\mathbb{G-}$subshift, $\mu$ a measure of maximal entropy of $X$%
	, $F \Subset \mathbb{G}$, and $w,v\in\mathcal{A}^{F}$. If $E(v)\subseteq E(w)$
	then 
	\begin{equation*}
	\mu(v)\leq\mu(w). 
	\end{equation*}
\end{Gtheorem}

As a direct consequence of this theorem we recover the following result due to Meyerovitch. 
\begin{theorem}[Theorem 3.1, \textrm{\protect\cite{M}}]
	\label{tomthm} If $X$ is a $\mathbb{Z}^{d}$-subshift and $v,w\in\mathcal{A}%
	^{F}$ satisfy $E_{X}(v)=E_{X}(w)$, then for every measure of maximal entropy 
	$\mu$ on $X$, $\mu(v)=\mu(w)$.
\end{theorem}

\begin{remark}
In fact the theorem from \cite{M} is more general; it treats equilibrium
states for a class of potentials $\phi$ with a property called $d$-summable
variation, and the statement here for measures of maximal entropy corresponds
to the $\phi=0$ case only.
\end{remark}

Due to our weaker hypothesis, $E_X(v) \subseteq E_X(w)$, our proof techniques are different from those used in \cite{M}. In particular, the case of different length $v,w$ treated in Theorem~\ref{hardcase} requires some subtle arguments about the ways in which $v,w$ can overlap themselves and each other.

Much as Corollary~\ref{maincor} was applicable to the class of synchronized subshifts, Theorem~\ref{Gtheorem} has new natural applications to the class of hereditary subshifts (introduced in \cite{KL1}), where there exist many pairs of words satisfying $E_{X}(v)\subsetneq E_{X}(w)$; see Section~\ref{hered} for details.

Section~\ref{defs} contains definitions and results needed throughout our
proofs, Section~\ref{zsec} contains our results for $\mathbb{Z}$-subshifts
(including various applications in Section~\ref{apps}), and Section~\ref%
{gsec} contains our results for $\mathbb{G}$-subshifts.

\section*{acknowledgments} 
We would like to thank the anonymous referee for their useful comments and suggestions.

\section{General definitions and preliminaries}

\label{defs}

We will use $\mathbb{G}$ to refer to a countable discrete group. We write 
$F \Subset \mathbb{G}$ to mean that $F$ is a finite subset of $\mathbb{G}$,
and unless otherwise stated, $F$ always refers to such an object.

A sequence $\{F_{n}\}_{n\in\mathbb{N}}$ with $F_{n} \Subset \mathbb{G}$ is
said to be \textbf{Følner} if for every $K\Subset\mathbb{G}$, we have that $%
|(K \cdot F_{n}) \Delta F_{n}|/|F_{n}|\rightarrow 0$. \ We say that $\mathbb{G}$
is \textbf{amenable} if it admits a Følner sequence. In particular, $\mathbb{Z%
}$ is an amenable group, since any sequence 
$\{F_{n}\} = [a_n, b_n] \cap \mathbb{Z}$ with $b_n - a_n \rightarrow \infty$
is F\o lner.

Let $\mathcal{A}$ be any finite set (usually known as the alphabet). We call 
$\mathcal{A}^{\mathbb{G}}$ the \textbf{full $\mathcal{A}$-shift on $\mathbb{G%
}$}, and endow it with the product topology (using the discrete topology on $%
\mathcal{A}$). For $x\in\mathcal{A}^{\mathbb{G}},$ we use $x_{i}$ to
represent the $i$th coordinate of $x$, and $x_{F}$ to represent the
restriction of $x$ to any $F\Subset\mathbb{G}$. 

For any $g \in\mathbb{G}$, we use $\sigma_{g}$ to denote the left
translation by $g$ on $\mathcal{A}^{\mathbb{G}}$, also called the \textbf{%
shift by $g$}; note that each $\sigma_{g}$ is an automorphism. We say $%
X \subseteq \mathcal{A}^{\mathbb{G}}$ is a $\mathbb{G}$-\textbf{subshift} if it
is closed and $\sigma_{g}(X) = X$ for all $g\in\mathbb{G}$; when $\mathbb{G=Z%
}$ we simply call it a subshift.

For $F \Subset\mathbb{G}$, we call an element of $\mathcal{A}^{F}$ a \textbf{%
word with shape $F$}. 
For $w$ a word with shape $F$ and $x$ either a point of $\mathcal{A}^{%
\mathbb{G}}$ or a word with shape $F^{\prime}\supset F$, we say that $w$ is
a \textbf{subword of $x$} if $x_{g + F} = w$ for some $g \in\mathbb{G}$.

For any $F$, the \textbf{$F$-language of $X$} is the set $L_{F}(X)\subseteq 
\mathcal{A}^{F}=\{x_{F}\ :\ x\in X\}$ of words with shape $F$ that appear as
subwords of points of $X.$ When $\mathbb{G} = \mathbb{Z}$, we use $L_{n}(X)$
to refer to $L_{\{0, \ldots, n-1\}}(X)$ for $n \in\mathbb{N}$. We define 
\begin{align*}
L(X) & :=\bigcup_{F\Subset\mathbb{G}}L_{F}(X)\text{ if }\mathbb{G\neq Z}%
\text{ and} \\
L(X) & :=\bigcup_{n \in\mathbb{N}}L_{n}(X)\text{ if }\mathbb{G=Z}\text{.}
\end{align*}

For any $\mathbb{G}$-subshift $X$ and $w\in L_{F}(X)$, we define the \textbf{%
cylinder set of $w$} as 
\begin{equation*}
\left[ w\right] :=\left\{ x\in X:x_{F}=w\right\} \text{.}
\end{equation*}



Whenever we refer to an interval in $\mathbb{Z}$, it means the intersection
of that interval with $\mathbb{Z}$. So, for instance, if $x\in\mathcal{A}^{%
	\mathbb{Z}}$ and $i < j$, $x_{\left[ i,j\right] }$ represents the subword of 
$x$ that starts in position $i$ and ends in position $j$. Unless otherwise
stated, a word $w\in\mathcal{A}^{n}$ is taken to have shape $[0, n)$. Every
word $w \in L(\mathcal{A}^{\mathbb{Z}})$ is in some $\mathcal{A}^{n}$ by
definition; we refer to this $n$ as the \textbf{length} of $w$ and denote it by $|w|$.

For any amenable $\mathbb{G}$ with Følner sequence $\left\{ F_{n}\right\}
_{n\in\mathbb{N}}$ and any $\mathbb{G}$-subshift $X$, we define the \textbf{%
topological entropy of $X$} as 
\begin{equation*}
h_{top}(X)=\lim_{n\rightarrow\infty}\frac{1}{\left\vert F_{n}
\right\vert }\log\left\vert L_{F_{n}}(X)\right\vert 
\end{equation*}
(this definition is in fact independent of the Følner sequence used.)

For any $w\in L(X)$, we define the \textbf{extender set of $w$} as 
\begin{equation*}
E_{X}(w):=\{x|_{F^{c}}\ :\ x\in\lbrack w]\}. 
\end{equation*}

\begin{example}
For any $\mathbb{G}$, if $X$ is the full shift on two symbols,
$\left\{ 0,1\right\}^\mathbb{G}$, then for any $F$, 
all words in $\{0,1\}^F$ have the same extender set, namely $\{0,1\}^{F^c}$.
\end{example}

\begin{example}
Take $\mathbb{G} = \mathbb{Z}^2$ and $X$ the hard-square shift on $\{0,1\}$
in which adjacent $1$s are forbidden horizontally and vertically. 
Then if we take $F = \{(0,0)\}$, we see that $E(0)$ is the set
of all configurations on $\mathbb{Z}^2 \setminus F$ which are legal, i.e.
which contain no adjacent $1$s. Similarly, $E(1)$ is the set of
all legal configurations on $\mathbb{Z}^2 \setminus F$ which also 
contain $0$s at $(0, \pm 1)$ and $(\pm 1, 0)$. In particular, we note that
here $E(1) \subsetneq E(0)$. 
\end{example}

In the specific case $\mathbb{G}=\mathbb{Z}$ and $w\in L_{n}(X)$, we may
identify $E_{X}(w)$ with the set of sequences which are concatenations of
the left side and the right side, i.e. $\{(x_{(-\infty,0)}x_{[n,\infty)})\
:\ x\in\lbrack w]\}$, and in this way can relate extender sets even for $v,w$
with different lengths. All extender sets in $\mathbb{Z}$ will be
interpreted in this way.


\begin{example}
If $X$ is the golden mean $\mathbb{Z-}$subshift on $\{0,1\}$ where adjacent $1$s are prohibited, then
$E(000)$ is the set of all legal configurations on $\mathbb{Z} \setminus \{0, 1, 2\}$, which is identified
with the set of all $\{0,1\}$ sequences $x$ which have no adjacent $1$s, with the exception that $x_{0} = x_{1} = 1$ is allowed.
This is because $000$ may be preceded by a one-sided sequence ending with $1$ and followed by a one-sided sequence beginning with $1$,
and after the identification with $\{0,1\}^{\mathbb{Z}}$, those $1$s could become adjacent.

Similarly, $E(01)$ is identified with the set of all $x$ on $\mathbb{Z}$ which have no adjacent $1$s and satisfy $x_{0} = 0$, and
$E(1)$ is identified with the set of all $x$ on $\mathbb{Z}$ which have no adjacent $1$s and satisfy $x_{0} = x_{1} = 0$. 

Therefore, even though they have different lengths, we can say here that $E(1) \subsetneq E(01) \subsetneq E(000)=E(0)$.
\end{example}




The next few definitions concern measures. Every measure in this work is
assumed to be a Borel probability measure $\mu$ on a $\mathbb{G-}$subshift $X
$ which is invariant under all shifts $\sigma_{g}$. By a generalization of
the Bogolyubov-Krylov theorem, every $\mathbb{G-}$subshift $X$ has at least
one such measure. For any such $\mu$ and any $w\in L(X)$, we will use $\mu(w)
$ to denote $\mu(\left[ w\right] ).$

For any Følner sequence $\{F_{n}\}$, we define the \textbf{entropy} of any
such $\mu$ as 
\begin{equation*}
h_{\mu}(X):=\lim_{n\rightarrow\infty}\frac{1}{\left\vert F_{n}\right\vert }%
\sum\nolimits_{w\in\mathcal{A}^{F_{n}}}-\mu(w)\log\mu(w). 
\end{equation*}

Again, this limit does not depend on the choice of Følner sequence (see \cite%
{KL2} for proofs of this property and of other basic properties of entropy
of amenable group actions).

It is always the case that $h_{\mu}(X)\leq h(X)$, and so a measure $\mu$ is
called a \textbf{measure of maximal entropy} (or \textbf{MME}) if $%
h_{\mu}(X)=h_{top}(X).$ For amenable $\mathbb{G}$, every $\mathbb{G}-$%
subshift has at least one measure of maximal entropy \cite{Mi}.

We briefly summarize some classical results from ergodic theory. A measure $%
\mu$ is \textbf{ergodic} if every set which is invariant under all $\sigma_{g}$
has measure $0$ or $1$. In fact, every measure $\mu$ can be written as a
generalized convex combination (really an integral) of ergodic measures;
this is known as the \textbf{ergodic decomposition} (e.g. see Section 8.7 of \cite{Gl}). The entropy map $%
\mu\mapsto h_{\mu}$ is linear and so the ergodic decomposition extends to
measures of maximal entropy as well; every MME can be written as a
generalized convex combination of ergodic MMEs. 

\begin{theorem}[Pointwise ergodic theorem \protect\cite{Li}]
\label{ergthm} For any ergodic measure $\mu$ on a $\mathbb{G}-$subshift $X$,
there exists a Følner sequence $\left\{ F_{n}\right\} $ such that for every $%
f\in L^{1}(\mu)$, 
\begin{equation*}
\mu\left( \left\{ x:\lim_{n\rightarrow\infty}\frac{1}{\left\vert
F_{n}\right\vert }\sum_{g\in F_{n}}f(\sigma_{g}x)=\int f\ d\mu\right\}
\right) =1. 
\end{equation*}
\end{theorem}

\begin{theorem}[Shannon-Macmillan-Breiman theorem for amenable groups \protect\cite{We}]
\label{SMBthm} For any ergodic measure $\mu$ on a $\mathbb{G}-$subshift $X$,
there exists a Følner sequence $\left\{ F_{n}\right\} $ such that 
\begin{equation*}
\mu\left( \left\{ x:\lim_{n\rightarrow\infty}-\frac{1}{\left\vert
F_{n}\right\vert } \log \mu(x_{F_{n}})=h_{\mu}(X)\right\} \right) =1. 
\end{equation*}
\end{theorem}

The classical pointwise ergodic and Shannon-Macmillan-Breiman theorems were originally stated for $\mathbb{G=Z}$ and the Følner sequence $[0,n]$. We only need Theorem~\ref{SMBthm} for the following
corollary (when $\mathbb{G=Z}$ this is essentially what is known as Katok's entropy formula; see \cite{Ka}).

\begin{corollary}
\label{SMBcor} Let $\mu$ be an ergodic measure of maximal entropy on a $%
\mathbb{G}$-subshift $X$. There exists a Følner sequence $\{F_{n}\}$ such that for every $S_{n}
\subseteq L_{F_{n}}(X)$ such that $\mu(S_{n}) \rightarrow1$, then 
\begin{equation*}
\lim_{n \rightarrow\infty} \frac{1}{|F_{n}|} \log|S_{n}| = h_{top}(X). 
\end{equation*}
\end{corollary}

\begin{proof}
Take $X$, $\mu$ as in the theorem, $\{F_{n}\}$ a Følner sequence that satisfies the Shannon-Macmillan-Breiman theorem, and $S_{n}$ as in the theorem. Fix any $%
\epsilon> 0$. By the definition of topological entropy, 
\begin{equation*}
\limsup_{n \rightarrow\infty} \frac{1}{|F_{n}|} \log|S_{n}| \leq\lim_{n
\rightarrow\infty} \frac{1}{|F_{n}|} \log|L_{F_{n}}(X)| = h_{top}(X). 
\end{equation*}

For every $n$, define 
\begin{equation*}
T_{n} = \{w \in\mathcal{A}^{F_{n}} \ : \ \mu(w) < e^{-|F_{n}|(h_{top}(X) -
\epsilon)} \}. 
\end{equation*}

By the Shannon-Macmillan-Breiman theorem, $\mu\left( \bigcup_{N} \bigcap_{n
= N}^{\infty} T_{n} \right) = 1$, and so $\mu(T_{n}) \rightarrow1$.
Therefore, $\mu(S_{n} \cap T_{n}) \rightarrow1$, and by definition of $T_{n}$%
, 
\begin{equation*}
|S_{n} \cap T_{n}| \geq\mu(S_{n} \cap T_{n}) e^{|F_{n}| (h_{top}(X) -
\epsilon)}. 
\end{equation*}
Therefore, for sufficiently large $n$, $|S_{n}| \geq|S_{n} \cap T_{n}|
\geq0.5 e^{|F_{n}| (h_{top}(X) - \epsilon)}$. Since $\epsilon> 0$ was
arbitrary, the proof is complete.
\end{proof}

Finally, several of our main arguments rely on the following elementary
combinatorial lemma, whose proof we leave to the reader.

\begin{lemma}
\label{counting} If $S$ is a finite set, $\{A_{s}\}$ is a collection of
finite sets, $m = \min\{|A_{s}|\}$, and $M = \max_{a \in\bigcup A_{s}} |\{s
\ | \ a \in A_{s}\}|$, then 
\begin{equation*}
\left| \bigcup_{s \in S} A_{s} \right| \geq|S| \frac{m}{M}. 
\end{equation*}
\end{lemma}

\section{Results on $\mathbb{Z-}$Subshifts}

\label{zsec}

\bigskip In this section we present the results for $\mathbb{G}=\mathbb{Z}$,
and must begin with some standard definitions about $\mathbb{Z}-$subshifts.


For words $v \in\mathcal{A}^{m}$ and $w \in\mathcal{A}^{n}$ with $m \leq n$,
we say that $v$ is a \textbf{prefix} of $w$ if $w_{[0,m)} = v$, and $v$ is a 
\textbf{suffix} of $w$ if $w_{[n-m, n)} = v$. 
\subsection{Main result}
\bigskip We now need some technical definitions about replacing one or more
occurrences of a word $v$ by a word $w$ inside a larger word $u$, which are
key to most of our arguments in this section. First, for any $v\in L(%
\mathcal{A}^{\mathbb{Z}}),$ we define the function $O_{v} :L(\mathcal{A}^{%
\mathbb{Z}}) \rightarrow \mathcal{P}(\mathbb{N})$ which sends any word $u$ to the set of
locations where $v$ occurs as a subword in $u$, i.e. 
\begin{equation*}
O_{v}(u):=\left\{ i\in\mathbb{N}:\sigma_{i}(u)\in\left[ v\right] \right\} . 
\end{equation*}
For any $w \in L(\mathcal{A}^{\mathbb{Z}})$, we may then define the function 
$R_{u}^{v\rightarrow w}: O_{v}(u) \rightarrow L(\mathcal{A}^{\mathbb{Z}})$
which replaces the occurrence of $v$ within $u$ at some position in $O_{v}(u)
$ by the word $w$. Formally, $R_{u}^{v\rightarrow w}(i)$ is the word $%
u^{\prime }$ of length $|u| - |v| + |w|$ defined by $u^{\prime}_{[0,i)} =
u_{[0,i)}$, $u^{\prime}_{[i,i+|w|)} = w$, and $u^{%
\prime}_{[i+|w|,|u|-|v|+|w|)} = u_{[i+|v|,|u|)}$.

\bigskip Our arguments in fact require replacing many occurrences of $v$ by $%
w$ within a word $u$, at which point some technical obstructions appear. For
instance, if several occurrences of $v$ overlap in $u$, then replacing one
by $w$ may destroy the other. The following defines conditions on $v$ and $w$
which obviate these and other problems which would otherwise appear in our
counting arguments.

\begin{definition}
\label{respect} For $v,w \in L(\mathcal{A}^{\mathbb{Z}})$, we say that $v$ 
\textbf{respects the transition to} $w$ if, for any $u\in L(\mathcal{A}^{%
\mathbb{Z}})$ and any $i\in O_{v}(u)$, 
\begin{align*}
\mathrm{(i) } \ & j+|w|-|v| \in O_{v}(R_{u}^{v\rightarrow w}(i))\text{ for
any }j\in O_{v}(u)\text{ with }i < j, \\
\mathrm{(ii) } \ & j \in O_{v}(R_{u}^{v\rightarrow w}(i))\text{ for any }%
j\in O_{v}(u)\text{ with }i>j, \\
\mathrm{(iii) } \ & j \in O_{w}(R_{u}^{v\rightarrow w}(i))\text{ for any }%
j\in O_{w}(u)\text{ with }i>j, \\
\mathrm{(iv) } \ & j + |w| - |v| > i \text{ for any }j\in O_{v}(u)\text{
with }i < j.
\end{align*}
\end{definition}

Informally, $v$ respects the transition to $w$ if, whenever a single
occurrence of $v$ is replaced by $w$ in a word $u$, all other occurrences of 
$v$ in $u$ are unchanged, all occurrences of $w$ in $u$ to the left of the
replacement are unchanged, and all occurrences of $v$ in $u$ which were to
the right of the replaced occurrence remain on that side of the replaced
occurrence.

When $v$ respects the transition to $w$, we are able to meaningfully define
replacement of a set of occurrences of $v$ by $w$, even when those
occurrences of $v$ overlap, as long as we move from left to right. For any $%
u,v,w \in L(\mathcal{A}^{\mathbb{Z}})$, we define a function $R_{u}^{v\rightarrow
w}: \mathcal{P}(O_{v}(u)) \rightarrow L(\mathcal{A}^{\mathbb{Z}})$ as follows. For any $S:=\left\{ s_{1}
,...,s_{n}\right\} \subseteq O_{v}(u)$ (where we always assume $s_1 < s_2 < \ldots < s_n$),
we define sequential replacements $%
\left\{ u^{m}\right\} _{m=1}^{n+1}$ by\ 

1) $u=u^{1}.$

2) $u^{m+1}=R_{u^{m}}^{v\rightarrow w}(s_{m}+(m-1)(|w|-|v|)).$

Finally, we define $R_{u}^{v\rightarrow w}(S)$ to be $u^{n+1}$.


\bigskip We first need some simple facts about $R_{u}^{v\rightarrow w}$ which are
consequences of Definition~\ref{respect}.

\begin{lemma}
\label{wsurvive} For any $u,v,w \in L(\mathcal{A}^{\mathbb{Z}})$ where $v$ respects the transition
to $w$ and any $S=\left\{ s_{1},...,s_{n}\right\} \subseteq O_{v}(u)$, all
replacements of $v$ by $w$ persist throughout, i.e. $\{s_{1},s_{2}
+(|w|-|v|),s_{3}+2(|w|-|v|),\ldots,s_{n}+(n-1)(|w|-|v|)\}\subseteq O_{w}
(R_{u}^{v\rightarrow w}(S))$.
\end{lemma}

\begin{proof}
Choose any $v,w,u,S$ as in the lemma, and any $s_{i} \in S$. Using the
terminology above, clearly $s_{i} + (i-1)(|w| - |v|) \in O_{w}(u^{(i+1)})$.
By property (iv) of a respected transition, $s_{1} < s_{2} + |w| - |v| <
\ldots< s_{n} + (n-1)(|w| - |v|)$. Then, since $s_{i} + (i-1)(|w| - |v|) <
s_{j} + (j-1)(|w| - |v|)$ for $j > i$, by property (iii) of respected
transition, $s_{i} + (i-1)(|w| - |v|) \in O_{w}(u^{(j+1)})$ for all $j > i$,
and so $s_{i} + (i-1)(|w| - |v|) \in O_{w}(R_{u}^{v \rightarrow w}(S))$.
Since $i$ was arbitrary, this completes the proof.
\end{proof}

\begin{lemma}
\label{vsurvive} For any $u,v,w \in L(\mathcal{A}^{\mathbb{Z}})$ where $v$ respects the transition
to $w$ and any $S=\left\{ s_{1},...,s_{n}\right\} \subseteq O_{v}(u)$, any
occurrence of $v$ not explicitly replaced in the construction of $%
R_{u}^{v\rightarrow w}$ also persists, i.e. if $m\in O_{v}(u)\setminus S$
and $s_{i}<m<s_{i+1}$, then $m+i(|w|-|v|)\in O_{v}(R_{u}^{v\rightarrow w}(S))
$.
\end{lemma}

\begin{proof}
Choose any $v,w,u,S$ as in the lemma, and any $m \in O_{v}(u) \cap(s_{i},
s_{i+1})$ for some $i$. Using property (i) of a respected transition, a
simple induction implies that $m + j(|w| - |v|) \in O_{v}(u^{(j+1)})$ for
all $j \leq i$. By property (iv) of a respected transition, $m + i(|w| -
|v|) < s_{i+1} + i(|w| - |v|) < \ldots< s_{n} + (n-1)(|w| - |v|)$.
Therefore, using property (ii) of a respected transition allows a simple
induction which implies that $m + i(|w| - |v|) \in O_{v}(u^{(j+1)})$ for all 
$j > i$, and so $m + i(|w| - |v|) \in O_{v}(R_{u}^{v \rightarrow w}(S))$.
\end{proof}

We may now prove injectivity of $R_{u}^{v\rightarrow w}$ under some
additional hypotheses, which is key for our main proofs.

\begin{lemma}
\label{injective} Let $v,w\in L(\mathcal{A}^{\mathbb{Z}})$ such that $v$ respects the transition to 
$w$, $v$ is not a suffix of $w$, and $w$ is not a prefix of $v$. For any $%
u\in L(\mathcal{A}^{\mathbb{Z}})$ and $m$, $R_{u}^{v\rightarrow w}$ is injective on the set of $m$%
-element subsets of $O_{v}(u)$.
\end{lemma}

\begin{proof}
Assume that $v,w,u$ are as in the lemma, and choose $S=\left\{ s_{1}
,...,s_{m}\right\} \neq S^{\prime}=\left\{ s_{1}^{\prime},...,s_{m}^{\prime
}\right\} \subseteq O_{v}(u)$ with $|S|=|S^{\prime}|=m$.

We first treat the case where $|v| \geq|w|$, and recall that $w$ is not a
prefix of $v$. Since $S \neq S^{\prime}$, we can choose $i$ maximal so that $%
s_{j} = s^{\prime}_{j}$ for $j < i$. Then $s_{i} \neq s^{\prime}_{i}$; we
assume without loss of generality that $s_{i} < s^{\prime}_{i}$. Since $%
s_{i} \in S$, we know that $s_{i} \in O_{v}(u)$. Since $s^{\prime}_{i-1} =
s_{i-1} < s_{i} < s^{\prime}_{i}$, by Lemma~\ref{vsurvive} $s_{i} +
(i-1)(|w| - |v|) \in O_{v}(R_{u}^{v \rightarrow w}(S^{\prime}))$. Also, by
Lemma~\ref{wsurvive}, $s_{i} + (i-1)(|w| - |v|) \in O_{w}(R_{u}^{v
\rightarrow w}(S))$. Since $w$ is not a prefix of $v$, this means that $%
R_{u}^{v \rightarrow w}(S) \neq R_{u}^{v \rightarrow w}(S^{\prime})$,
completing the proof of injectivity in this case.

We now treat the case where $|v| \leq|w|$, and recall that $v$ is not a
suffix of $w$. Since $S \neq S^{\prime}$, we can choose $i$ maximal so that $%
s_{m - j} = s^{\prime}_{m - j}$ for $j < i$. Then $s_{m - i} \neq
s^{\prime}_{m - i} $; we assume without loss of generality that $s_{m - i} <
s^{\prime}_{m - i}$. Since $s^{\prime}_{m - i} \in S^{\prime}$, we know that 
$s^{\prime}_{m - i} \in O_{v}(u)$. Since $s_{m-i} < s^{\prime}_{m-i} <
s^{\prime}_{m-i+1} = s_{m-i+1}$, by Lemma~\ref{vsurvive} $s^{\prime}_{m-i} +
(m-i)(|w| - |v|) \in O_{v}(R_{u}^{v \rightarrow w}(S))$. Also, by Lemma~\ref%
{wsurvive}, $s^{\prime }_{m-i} + (m-i-1)(|w| - |v|) \in O_{w}(R_{u}^{v
\rightarrow w}(S^{\prime}))$. Since $v$ is not a suffix of $w$, this means
that $R_{u}^{v \rightarrow w}(S) \neq R_{u}^{v \rightarrow w}(S^{\prime})$,
completing the proof of injectivity in this case and in general.
\end{proof}

\begin{lemma}
\label{preimage} Let $v,w\in L(\mathcal{A}^{\mathbb{Z}})$ such that $v$ respects the transition to $%
w$, $v$ is not a suffix of $w$, and $w$ is not a prefix of $v$. Then for any 
$u^{\prime}\in L(\mathcal{A}^{\mathbb{Z}})$ and any $m \leq|O_{w}(u^{\prime})|$, 
\begin{equation*}
|\{(u, S) \ : \ |S| = m, S \subseteq O_{v}(u), u^{\prime}=
R_{u}^{v\rightarrow w}(S)\}| \leq{\binom{|O_{w}(u^{\prime})| }{m}}. 
\end{equation*}
\end{lemma}

\begin{proof}
Assume that $v,w,u^{\prime}$ are as in the lemma, and denote the set above
by $f(u^{\prime})$. For any $(u,S)\in f(u^{\prime})$ we define $g(S)=\{s_{1}
,s_{2}+|w|-|v|,\ldots,s_{m}+(m-1)(|w|-|v|)\}$; note that by Lemma~\ref%
{wsurvive}, $g(S)\subseteq O_{w}(u^{\prime})$.

We claim that for any $S$, there is at most one $u$ for which $(u,S)\in
f(u^{\prime})$. One can find this $u$ by simply reversing each of the
replacements in the definition of $R_{u}^{v \rightarrow w}(S)$. Informally,
the only such $u$ is $u = R_{u^{\prime}}^{v \leftarrow w}(g(S))$, where $%
R_{u'}^{v\leftarrow w}$ is defined analogously to $R_{u}^{v\rightarrow w}$
with replacements of $w$ by $v$ made from right to left instead of $v$ by $w$ made from left to right.

Finally, since $g(S)\subseteq O_{w}(u^{\prime})$, and since $g$ is clearly
injective, there are less than or equal to ${\binom{|O_{w}(u^{\prime})|}{m}}$
choices for $S$ with $(u,S)\in f(u^{\prime})$ for some $u$, completing the
proof.
\end{proof}

We may now prove the desired relation for $v$, $w$ with $E(v)\subseteq E(w)$
under additional assumptions on $v$ and $w$. 

\begin{proposition}
\label{easycase} Let $X$ be a subshift, $\mu$ a measure of maximal entropy
of $X$, and $v,w\in L(X).$ If $v$ respects the transition to $w$, $v$ is not
a suffix of $w$, $w$ is not a prefix of $v$, and $E_{X}(v)\subseteq E_{X}(w)$%
, then 
\begin{equation*}
\mu(v)\leq\mu(w)e^{h_{top}(X)(|w|-|v|)}. 
\end{equation*}
\end{proposition}

\begin{proof}
Let $\delta,\varepsilon\in\mathbb{Q}_{+}$. We may assume without loss of
generality that $\mu$ is an ergodic MME, since proving the desired
inequality for ergodic MMEs implies it for all MMEs by ergodic
decomposition. 

For every $n\in\mathbb{Z}_{+},$ we define 
\begin{equation*}
S_{n}:=\left\{ u\in L_{n}(X):\left\vert O_{v}(u)\right\vert \geq
n(\mu(v)-\delta)\text{ and }\left\vert O_{w}(u)\right\vert \leq n(\mu
(w)+\delta)\right\} . 
\end{equation*}

By the pointwise ergodic theorem (applied to $\chi_{[v]}$ and $\chi_{[w]}$), 
$\mu(S_{n}) \rightarrow1$. Then, by Corollary~\ref{SMBcor}, there exists $N$
so that for $n>N$, 
\begin{equation}
|S_{n}|>e^{n(h_{top}(X)-\delta)}.   \label{Sbound}
\end{equation}
For each $u\in S_{n}$, we define 
\begin{equation*}
A_{u}:=\left\{ R_{u}^{v\rightarrow w}(S):S\subseteq O_{v}(u)\text{ and }%
\left\vert S\right\vert =\varepsilon n\right\} 
\end{equation*}
(without loss of generality we may assume $\varepsilon n$ is an integer by
taking a sufficiently large $n$.)

Since each word in $A_u$ is obtained by making $\varepsilon n$ replacements
of $v$ by $w$ in a word of length $n$, all words in $A_u$ have length
$m := n + \varepsilon n (|w| - |v|)$.
Since $E_{X}(v)\subseteq E_{X}(w)$, we have that $A_{u}\subset L(X)$. Also,
by Lemma~\ref{injective}, 
\begin{equation*}
|A_{u}|={\binom{\left\vert O_{v}(u)\right\vert}{\left\vert S\right\vert} \geq {\binom{n(\mu(v)-\delta)}{\varepsilon n}} }
\end{equation*}
for every $u$.

On the other hand, for every $u^{\prime}\in\bigcup_{u\in S_{n}}A_{u}$ we
have that 
\begin{equation*}
\left\vert O_{w}(u^{\prime})\right\vert \leq n(\mu(w)+\delta)+n\varepsilon
(2|w|+1) 
\end{equation*}
(here, we use the fact that any replacement of $v$ by $w$ can create no more
than $2|w|$ \ new occurrences of $w$.) Therefore, by Lemma~\ref{preimage}, 
\begin{equation*}
\left\vert \left\{ u\in S_{n}:u^{\prime}\in A_{u}\right\} \right\vert \leq{%
\binom{n\left(\mu(w)+\delta+(2|w|+1)\varepsilon\right)}{\varepsilon n}.} 
\end{equation*}

Then, by Lemma~\ref{counting}, we see that for $n>N$, 
\begin{multline}
|L_{m}(X)|\geq\left\vert \bigcup_{u\in S_{n}}A_{u}\right\vert \geq |S_{n}|{%
\binom{n(\mu(v)-\delta)}{\varepsilon n}}{\binom{n(\mu(w)+\delta
+(2|w|+1)\varepsilon)}{\varepsilon n}}^{-1}  \label{imagebound} \\
\geq e^{n(h_{top}(X)-\delta)}{\binom{n(\mu(v)-\delta)}{\varepsilon n}} {%
\binom{n(\mu(w)+\delta+(2|w|+1)\varepsilon)}{\varepsilon n}}^{-1}.
\end{multline}

For readability, we define $x=\mu(v)-\delta$ and $y=\mu(w)+\delta$. 
We recall that by Stirling's approximation, for $a > b > 0$,
\[
\log\left(
\begin{array}
[c]{c}%
an\\
bn
\end{array}
\right)  =an\log(an)-bn\log(bn)-n(a-b)\log(n(a-b))+o(n)
.\]
Therefore, if we take logarithms and divide by $n$ on both sides of (\ref{imagebound}) and let $n$ approach infinity, we obtain
\begin{multline*}
h_{top}(X)(1+\varepsilon(|w|-|v|))\geq h_{top}(X)-\delta+x\log
x-(x-\varepsilon)\log(x-\varepsilon) \\
-(y+(2|w|+1)\varepsilon)\log(y+(2|w|+1)\varepsilon)+(y+2|w|\varepsilon
)\log(y+2|w|\varepsilon).
\end{multline*}
We subtract $h_{top}(X)$ from both sides, let $\delta\rightarrow0$, and
simplify to obtain 
\begin{multline*}
h_{top}(X)\varepsilon(|w|-|v|)\geq\varepsilon\log\mu(v)+(\mu(v)-\varepsilon
)\left( \log\frac{\mu(v)}{\mu(v)-\varepsilon}\right) \\
-\varepsilon\log(\mu(w)+(2|w|+1)\varepsilon)-(\mu(w)+2|w|\varepsilon
)\log\left( \frac{\mu(w)+(2|w|+1)\varepsilon}{\mu(w)+2|w|\varepsilon}\right)
.
\end{multline*}

We have that 
\begin{align*}
& \lim_{\varepsilon\rightarrow0}\frac{\mu(v)-\varepsilon}{\varepsilon} \log%
\frac{\mu(v)}{\mu(v)-\varepsilon} \\
& =\lim_{\varepsilon\rightarrow0}\frac{\mu(v)}{\varepsilon}\log\frac{\mu (v)%
}{\mu(v)-\varepsilon} \\
& =1,
\end{align*}
and 
\begin{align*}
& \lim_{\varepsilon\rightarrow0}-\frac{\mu(w)+2|w|\varepsilon}{\varepsilon }%
\log\left( \frac{\mu(w)+(2|w|+1)\varepsilon}{\mu(w)+2|w|\varepsilon}\right)
\\
& =\lim_{\varepsilon\rightarrow0}-\frac{\mu(w)}{\varepsilon}\log\left( \frac{%
\mu(w)+(2|w|+1)\varepsilon}{\mu(w)+2|w|\varepsilon}\right) \\
& =\lim_{t\rightarrow0}-\frac{1}{t}\log\left( \frac{1+(2|w|+1)t} {1+2|w|t}%
\right) \\
& =-1.
\end{align*}

This implies (by dividing by $\varepsilon$ and taking limit on the previous
estimate) that 
\begin{equation*}
h_{top}(X)(|w|-|v|)\geq\log\mu(v)-\log\mu(w). 
\end{equation*}
Exponentiating both sides and solving for $\mu(v)$ completes the proof.
\end{proof}

Our strategy is now to show that any pair $v, w$, the cylinder sets $[v]$
and $[w]$ may each be partitioned into cylinder sets of the form $[\alpha v
\beta]$ and $[\alpha w \beta]$ where the additional hypotheses of Theorem~%
\ref{easycase} hold on corresponding pairs. For this, we make the additional
assumption that $X$ has positive entropy to avoid some pathological examples
(for instance, note that if $X = \{0^{\infty}\}$, then it's not even possible
to satisfy the hypotheses of Theorem~\ref{easycase}!)


\begin{definition}
	Let $X$ be a subshift and $v\neq w\in L(X)$. We define
	\begin{align*}
	X_{resp(v\rightarrow w)}  & :=\{x\in\left[  v\right]  :\\
	\exists N,M  & \in\mathbb{Z}_{+} \text{ s.t. } \alpha v\beta=x_{[-N,M)}\text{ respects
		the transition to }\alpha w\beta,\\
	& \alpha v\beta\text{ is not a suffix of }\alpha w\beta\text{, and }\\
	& \alpha w\beta\text{ is not a prefix of }\alpha v\beta\}.
	\end{align*}
\end{definition}
\begin{proposition}
\label{extend} \label{transition} Let $X$ be a subshift with positive topological entropy, $\mu$ an ergodic
measure of maximal entropy of $X$, and $v\neq w\in L(X)$. There exists $G^{v,w}\subset X_{resp(v\rightarrow w)}$ such that $\mu(G^{v,w})=\mu(v)$. \\

\end{proposition}

\begin{proof}
Define 
\begin{equation*}
Q:=\left\{ \gamma\in L(X):\mu(\gamma)>0\right\} 
\end{equation*}
and, for all $n \in\mathbb{N}$, define $Q_{n} := Q \cap A^{n}$.

Recall that 
\begin{equation*}
h_{\mu}(X)=\lim_{n\rightarrow\infty}\frac{1}{n}\sum\nolimits_{w\in
A^{n}}-\mu(w)\log\mu(w). 
\end{equation*}
The only positive terms of this sum are those corresponding to $w \in Q_{n}$%
, and it's a simple exercise to show that when $\sum_{i=1}^{t} \alpha_{i} = 1
$, $\sum_{i=1}^{t} -\alpha_{i} \log\alpha_{i}$ has a maximum value of $\log t
$. Therefore, 
\begin{equation*}
h_{\mu}(X) \leq\liminf_{n \rightarrow\infty} \frac{1}{n} \log|Q_{n}|. 
\end{equation*}
Since $h_{\mu}(X) > 0$, $\left\vert Q_{n}\right\vert $ grows exponentially.
Therefore, there exists $n_{2}\in\mathbb{Z}_{+}$ such that for every $n\geq
n_{2}$ we have that $\left\vert Q_{n}\right\vert $ $\geq2n.$

Let 
\begin{align*}
N & :=\max\left\{ n_{2},\left\vert v\right\vert \right\} +1, \\
P & :=\left\{ x \in X \ : \ x_{(-\infty, 0)} \text{ periodic with period
less than } |w| \right\} , \\
S & :=\left\{ x \in X \ : \ \forall\gamma\in Q, \gamma\text{ is a subword of 
} x_{[0, \infty)} \right\} ,\text{ and} \\
G^{v,w} & :=[v]\cap S \diagdown P.
\end{align*}

 Since $\mu$ has
positive entropy, it is not supported on points with period less than $|w|$,
and so for each $i \leq|w|$, there exists a word $u_{i} \in L_{i+1}(X)$ with
different first and last letters. Then the pointwise ergodic theorem
(applied to $\chi_{[u_{1}]}, \ldots, \chi_{[u_{|w| - 1}]}$ with $F_n = [-n,0)$) implies that $%
\mu(P) = 0$. The pointwise ergodic theorem (applied to $\chi_{[\gamma]}$ for 
$\gamma\in Q$ with $F_n = [0,n]$) shows that $\mu(S)=1$, and so $\mu(G^{v,w})=\mu(v)$. \\

Now we will prove that $G^{v,w}\subset X_{resp(v\rightarrow w)}$.
Let $x\in R$. If for every $n$, $x_{(-n,0)}v$ is a suffix of $x_{(-n,0)}w$,
then clearly $|w|\geq |v|$, and for any $i>0$, the $(i+|w|)$th letters from the
end of $x_{(-\infty,0)}v$ and $x_{(-\infty,0)}w$ must be the same, i.e. $%
x(-i)=x(-i-|w|+|v|)$. This would imply $x\in P$, which is not possible.

We can therefore define $N^{\prime}\geq N$ to be minimal so that for $%
\alpha^x= x_{[-N^{\prime}, 0)}$, $\alpha^x v$ is not a suffix of $\alpha^x w$.
(Obviously if $|v| \geq|w|$, then $N^{\prime}= N$.)

Since $x \in S$, we can define the minimal $M$ so that all $N^{\prime} $%
-letter words of positive $\mu$-measure are subwords of $x_{\left[
-N^{\prime},M\right) }$; for brevity we write this as $Q_{N^{\prime}}
\sqsubset x_{\left[ -N^{\prime},M\right) }$.

Since $M$ is the first
natural with $Q_{N^{\prime}}\sqsubset x_{\left[ -N^{\prime},M\right) }$,
then 
\begin{equation*}
\left\vert O_{x_{\left[ M-N^{\prime},M\right) }}(x_{\left[ -N^{\prime
},M\right) })\right\vert =1, 
\end{equation*}
i.e. the $N^{\prime}$-letter suffix of $x_{[-N^{\prime}, M)} = \alpha^x v \beta^x
$ appears only at the end of $\alpha^x v \beta^x$. 
Since $N^{\prime}\geq N \geq n_{2}$, $\left\vert Q_{N^{\prime}}\right\vert $ 
$\geq 2N^{\prime}$, and so $M > 2N^{\prime} \geq N^{\prime} + |v|$, implying that
the aforementioned $N^{\prime}$-letter suffix of $\alpha^x v \beta^x$ is 
also the $N^{\prime}$-letter suffix of $\alpha^x w \beta^x$.

First, it is clear that $\alpha^x v \beta^x$ is not a suffix of $\alpha^x w \beta^x$
, since $\alpha^x v$ was not a suffix of $\alpha^x w$ by definition of $\alpha^x$.
Since the $N^{\prime}$-letter suffix of $\alpha^x w \beta^x$ appears only once
within $\alpha^x v \beta^x$, we see that $\alpha^x w \beta^x$ cannot be a prefix of $%
\alpha^x v \beta^x$ either.

It remains to show that $\alpha^x v\beta^x=x_{\left[ -N^{\prime},M\right) }$
respects the transition to $\alpha^x w\beta^x.$ Suppose that a word $u\in L(X)$
contains overlapping copies of $\alpha^x v \beta^x$, i.e. we have $i,j\in
O_{\alpha^x v\beta^x}(u)$ with $j>i$. Since $\left\vert O_{x_{\left[ M-N^{\prime
},M\right) }}(x_{\left[ -N^{\prime},M\right) })\right\vert =1$ we have that $%
j > i + M$; otherwise the $N^{\prime}$-letter suffix of $\alpha^x v
\beta^x= x_{[i, i+N^{\prime}+M)}$ would be a non-terminal subword of $%
\alpha^x v \beta^x= x_{[j, j+N^{\prime}+M)}$. Then $j + |w| - |v| > 
i + M + |w| - |v| > i$, and so property (iv) is verified. Since 
$j > i + M$, the
central $v$ within $x_{[i, i+N^{\prime}+M})$ is disjoint from $x_{[j,
j+N^{\prime }+M})$, and so $j + |w| - |v| \in O_{v}(R_{u}^{v\rightarrow
w}(i))$, verifying property (i).

For property (ii), the same argument as above shows that when $i,j\in
O_{\alpha^x v\beta^x}(u)$ with $i>j$, $i > j + M$. Again this means that
the central $v$ within $x_{[i, i+N^{\prime}+M)}$ is disjoint from $%
x_{[j, j+N^{\prime}+M)}$, and so $j \in O_{v}(R_{u}^{v\rightarrow w}(i))$%
, verifying property (ii) and completing the proof.

For property (iii), we simply note that the proof of (ii) is completely
unchanged if we instead assumed $j \in O_{\alpha^x w\beta^x}(u)$, since the $%
N^{\prime}$-letter suffixes of $\alpha^x w \beta^x$ and $\alpha^x v \beta^x$ are the
same.

\end{proof}

\begin{remark}\label{cp}
For $x\in G^{v,w}$ (as in Proposition~\ref{extend}) we denote by $\alpha^{x}$ and $\beta
^{x}$ the words $\alpha$ and $\beta$ constructed in the proof. 	
\end{remark}

\begin{lemma}\label{cl}
For $x \neq y\in G^{v,w}$, it is not possible for either of $\alpha^x$, $\alpha^y$ to be a proper suffix of the other, and
if $\alpha^x = \alpha^y$, then it is not possible for either of $\beta^x, \beta^y$ to be a proper prefix of the other.
\end{lemma}

\begin{proof}
	Let $x \neq y\in G^{v,w}$. We write $\alpha^x v \beta^x = x_{[-N'_x, M_x)}$ and $\alpha^y v \beta^y= y_{[-N'_y, M_y)}$. We recall that 
$\alpha^x= x_{[-N'_x, 0)}$ was chosen as the minimal $N'_x$ 
(above a certain $N_x$ dependent only on $v$ and $X$) so that $\alpha^x v$ is not a suffix of $\alpha^x w$,
and that $\alpha^y = y_{[-N'_y, 0)}$ was defined similarly using minimal $N'_y$ above some $N_y$.
If $\alpha^y$ were a proper suffix of $\alpha^x$, then $N'_y < N'_x$ and 
$\alpha^y = x_{[-N'_y, 0)}$. Since by construction $\alpha^y v$ is not a suffix of $\alpha^y w$, this would contradict the minimality of $\alpha^x$.
A trivially similar argument shows that $\alpha^x$ is not a proper suffix of $\alpha^y$.

Now, assume that $\alpha^x = \alpha^y$; we denote their common value by $\alpha$ and their common length by 
$N'$. Recall that $\beta^x = x_{[|v|, M_x)}$ was chosen using the minimal
$M_x$ so that $\alpha^x v \beta^x$ contains all $N'_x$-letter words of positive $\mu$-measure, and 
that $\beta^y$ was defined similarly using minimal $M_y$ for $y$.
If $\beta^y$ were a proper prefix of $\beta^x$, then $M_y < M_x$ and 
$\beta^y = x_{[|v|, M_y)}$. Since $\alpha v \beta^y$ contains all $N'$-letter words of positive $\mu$-measure, 
this would contradict the minimality of $\beta^x$. A trivially similar argument shows that $\beta^x$ is not a proper prefix of 
$\beta^y$.

\end{proof}

We may now prove the main result of this section.

\begin{theorem}
\label{hardcase} Consider any $X$ a subshift with positive entropy, $\mu$ a measure of maximal entropy of 
$X$, and $v,w\in L(X).$ If $E_{X}(v)\subseteq E_{X}(w)$ then 
\begin{equation*}
\mu(v)\leq\mu(w)e^{h_{top}(X)(|w|-|v|)}. 
\end{equation*}
\end{theorem}

\begin{proof}
Consider $X, \mu, v, w$ as in the theorem. We may prove the result for only ergodic $\mu$, since it then follows for
all $\mu$ by ergodic decomposition. 

If $v=w$ the result is trivial, so we assume $v\neq w$. Let  $G^{v,w}$ be as in the proof of Proposition~\ref{extend}. 

For any $x \in G^{v,w}$, by definition $\alpha^x v \beta^x \in L(X)$. Since $E_{X}(v)\subseteq E_{X}
(w)$, we then know that $\alpha^x w \beta^x\in L(X)$ and $E_{X}(\alpha^x v \beta^x)\subseteq E_{X}(\alpha^x w \beta^x)$
for every $x\in G^{v,w}$. 
Now, using Proposition~\ref{easycase} we
have that 
\begin{equation}  \label{hardbound}
\mu(\alpha^x v \beta^x)\leq\mu(\alpha^x w \beta^x)e^{h_{top}(X)(|\alpha^x w \beta^x|-|\alpha^x v \beta^x|)}=\mu(\alpha^x w \beta^x)e^{h_{top}(X)(|w|-|v|)}.
\end{equation}

For convenience, we adopt the notation $[\alpha^x v \beta^x] = [\alpha^x.v\beta^x]$ and
$[\alpha^x w \beta^x] = [\alpha^x.w\beta^x]$ to emphasize the location of the words $\alpha^x v \beta^x$ and $\alpha^x w \beta^x$ within
these cylinder sets.

We now claim that if $\alpha^x v \beta^x \neq \alpha^y v \beta^y$ for $x,y \in G^{v,w}$, then $[\alpha^x v \beta^x]\cap [\alpha^y v \beta^y]=\emptyset$. To verify this, choose any $x,y$ for which $\alpha^x v \beta^x \neq \alpha^y v \beta^y$; then either $\alpha^x\neq\alpha^{y}$ or $\alpha^x = \alpha^y$ and $\beta^x\neq\beta^{y}$.
If $\alpha^x \neq \alpha^y$, then by Lemma~\ref{cl}, neither of $\alpha^x$ or $\alpha^y$ can be a suffix of the other, which means that the cylinder sets $[\alpha^x .v \beta^x]$ and $[\alpha^{y} .v \beta^{y}]$ are disjoint.

If instead $\alpha^x = \alpha^y$ and $\beta^x \neq \beta^y$, 
then again by Lemma~\ref{cl}, neither of $\beta^x$ or $\beta^y$ can be a prefix of the other,
meaning that the cylinder sets $[\alpha^x .v \beta^x]$ and $[\alpha^{y} .v \beta^{y}]$ are again disjoint. This proves the claim.

Let $K=\{\alpha^x v \beta^x \ : \ x\in G^{v,w}\}$.
Since all $[\alpha^x .v \beta^x]$ are
disjoint or equal,
$\{[\alpha .v \beta]\}_{\alpha v \beta \in K}$ forms a partition of $G^{v,w}$. Furthermore we also obtain that the
sets $\{[\alpha^x .w \beta^x]\}_{\alpha v \beta \in K}$ are disjoint, and so 
\begin{align*}
\sum_{\alpha v \beta \in K}\mu(\alpha v \beta) & =\mu(G^{v,w})=\mu(v)\text{ and} \\
\sum_{\alpha v \beta \in K}\mu(\alpha w \beta) & \leq\mu(w).
\end{align*}
In fact one can show the final inequality is an equality but we will not use this. We
may then sum (\ref{hardbound}) over $\alpha v \beta \in K$ yielding 
\begin{align*}
\mu(v) & =\sum_{\alpha v \beta \in K}\mu(\alpha v \beta) \\
& \leq e^{h_{top}(X)(|w|-|v|)}\sum_{\alpha v \beta \in K}
\mu(\alpha w \beta) \\
& \leq\mu(w)e^{h_{top}(X)(|w|-|v|)},
\end{align*}
as desired.

\end{proof}

The following corollary is immediate.

\begin{corollary}
\label{maincor} Let $X$ be a $\mathbb{Z}$-subshift, $\mu$ a measure of
maximal entropy of $X$, and $w,v\in L(X).$ If $E_{X}(v)=E_{X}(w)$, then for
every measure of maximal entropy of $X$, 
\begin{equation*}
\mu(v)=\mu(w)e^{h_{top}(X)(|w|-|v|)}. 
\end{equation*}
\end{corollary}

\subsection{Applications to synchronized subshifts}
\label{apps}




The class of synchronized subshifts provides many examples where $E_X(v) = E_X(w)$ is satisfied for many pairs $v,w$ of different lengths, allowing for the usage of Corollary~\ref{maincor}.

\begin{definition}
	For a subshift $X$, we say that $v\in L(X)$ is \textbf{synchronizing} if for
	every $uv,vw\in L(X)$, it is true that $uvw\in L(X).$ A subshift $X$ is
	\textbf{synchronized} if $L(X)$ contains a synchronizing word.
\end{definition}


The following fact is immediate from the definition of synchronizing word.

\begin{lemma}
	\label{synchlem} If $w$ is a synchronizing word for a subshift $X$, then for
	any $v \in L(X)$ which contains $w$ as both a prefix and suffix, $E_{X}(v) =
	E_{X}(w)$.
\end{lemma}

\begin{definition}
	A subshift $X$ is \textbf{entropy minimal }if every subshift strictly
	contained in $X$ has lower topological entropy. Equivalently, $X$ is entropy
	minimal if every MME on $X$ is fully supported.
\end{definition}

The following result was first proved in \cite{Th}, but we may also derive it
as a consequence of Corollary~\ref{maincor} with a completely different proof.

\begin{theorem}
	\label{synchunique} Let $X$ be a synchronized subshift. If $X$ is entropy
	minimal then $X$ has a unique measure of maximal entropy.
\end{theorem}

\begin{proof}
	Let $\mu$ be an ergodic measure of maximal entropy of such an $X$. Let $w$ be
	a synchronizing word, $u\in L(X)$ and
	\[
	R_{u}:=\left\{  x\in\left[  u\right]  :\left\vert O_{w}(x_{\left(
		-\infty,0\right]  })\right\vert \geq1\text{ and }\left\vert O_{w}(x_{\left[
		\left(  \left\vert u\right\vert ,\infty\right]  \right)  })\right\vert
	\geq1\right\}  .
	\]
	Since $X$ is entropy minimal, $\mu(w) > 0$, and so by the pointwise ergodic
	theorem (applied to $\chi_{[w]}$ with $F_{n} = [-n,0]$ or $(|u|, n]$),
	$\mu(R_{u}) = \mu(u)$.
	
	For every $x\in R_{u}$ we define minimal $n \geq|w|$ and $m \geq|w| + |u|$ so
	that $g_{u}(x):=x_{\left[  -n,m\right]  }$ contains $w$ as both a prefix and a
	suffix. Then $\{[g_{u}(x)]\}$ forms a partition of $R_{u}$.
	
	By Lemma~\ref{synchlem}, $E_X(w) = E_X(wvw)$ for all $v$ s.t. $wvw \in L(X)$. Then
	by Corollary \ref{maincor} we have that
	\[
	\mu(g_{u}(x))=\mu(w)e^{h_{top}(X)(\left\vert w\right\vert -\left\vert
		g_{u}(x)\right\vert )}.
	\]
	Since $g_{u}(R_{u})$ is countable we can write
	\[
	\mu(u)=\mu(R_{u}) = \mu(w)\sum_{g_{u}(x)\in g_{u}(R_{u})}e^{h_{top}
		(X)(\left\vert w\right\vert -\left\vert g_{u}(x)\right\vert )}.
	\]

	This implies that
	\[
	1=\sum_{a\in\mathcal{A}}\mu(a)=\mu(w)\sum_{a\in\mathcal{A}}\sum_{g_{a}(x)\in
		g_{a}(R_{a})}e^{h_{top}(X)(\left\vert w\right\vert -\left\vert g_{a}
		(x)\right\vert )}.
	\]
	We combine the two equations to yield
	\begin{align*}
	\mu(u)  &  =\frac{\sum\nolimits_{g_{u}(x)\in g_{u}(R_{u})}e^{h_{top}
			(X)(\left\vert w\right\vert -\left\vert g_{u}(x)\right\vert )}}{\sum
		_{a\in\mathcal{A}}\sum_{g_{a}(x)\in g_{a}(R_{a})}e^{h_{top}(X)(\left\vert
			w\right\vert -\left\vert g_{a}(x)\right\vert )}}\\
	&  =\frac{\sum\nolimits_{g_{u}(x)\in g_{u}(R_{u})}e^{-h_{top}(X)\left\vert
			g_{u}(x)\right\vert }}{\sum_{a\in\mathcal{A}}\sum\nolimits_{g_{a}(x)\in
			g_{a}(R_{a})}e^{-h_{top}(X)\left\vert g_{a}(x)\right\vert }}.
	\end{align*}

	Since the right-hand side is independent of the choice of the measure we
	conclude there can only be one ergodic measure of maximal entropy, which
	implies by ergodic decomposition that there is only one measure of maximal entropy.
\end{proof}

In \cite{CT}, one of the main tools used in proving uniqueness of the measure
of maximal entropy for various subshifts was boundedness of the quantity
$\frac{|L_{n}(X)|}{e^{nh_{top}(X)}}$. One application of our results is to
show that this quantity in fact converges to a limit for a large class of
synchronized shifts.

\begin{definition}
	A measure $\mu$ on a subshift $X$ is \textbf{mixing} if, for all measurable
	$A,B$,
	\[
	\lim_{n \rightarrow\infty} \mu(A \cap\sigma_{-n} B) = \mu(A) \mu(B).
	\]
	
\end{definition}

\begin{theorem}
	\label{limit}Let $X$ be a synchronized entropy minimal subshift such that the
	measure of maximal entropy is mixing. We have that
	\[
	\lim_{n\rightarrow\infty}\frac{\left\vert L_{n}(X)\right\vert }{e^{nh_{top}
			(X)}}\text{ exists.}%
	\]
	
\end{theorem}

\begin{proof}
	We denote $\lambda:=e^{h_{top}(X)}$ and define $\mu$ to be the unique measure
	of maximal entropy for $X$. Let $w\in L(X)$ be a synchronizing word and
	\[
	R_{n}:=\left\{  u\in L_{n}(X):w\text{ is a prefix and a suffix of }u\right\}
	.
	\]
	Lemma~\ref{synchlem} and Corollary \ref{maincor} imply that for every $u\in
	R_{n}$,
	\[
	\mu(u)=\mu(w)\lambda^{\left\vert w\right\vert -n}.
	\]

	This implies that
	\[
	\sum_{u\in R_{n}}\mu(u)=\left\vert R_{n}\right\vert \mu(w)\lambda^{\left\vert
		w\right\vert -n}%
	\]

	On the other hand
	\[
	\sum_{u\in R_{n}}\mu(u)=\mu(\left[  w\right]  \cap\sigma_{\left\vert
		w\right\vert -n}\left[  w\right]  ).
	\]
	Since the measure is mixing we obtain that
	\[
	\lim_{n\rightarrow\infty}\mu(\left[  w\right]  \cap\sigma_{\left\vert
		w\right\vert -n}\left[  w\right]  )=\mu(\left[  w\right]  )^{2}.
	\]

	Combining the three equalities above yields
	\[
	\lim_{n\rightarrow\infty}\frac{\left\vert R_{n}\right\vert }{\lambda^{n}
	}=\frac{\mu(w)}{\lambda^{\left\vert w\right\vert }}.
	\]

	For all $n\in\mathbb{N}$, we define
	\begin{align*}
	P_{n}  &  :=\left\{  u\in L_{n+|w|}(x):w\text{ is a prefix of }u,|O_{w}|
	(u)=1\right\}  \text{ and}\\
	S_{n}  &  :=\left\{  u\in L_{n+|w|}(x):w\text{ is a suffix of }u,|O_{w}|
	(u)=1\right\}
	\end{align*}
	to be the sets of $(n+|w|)$-letter words in $L(X)$ containing $w$ exactly once
	as a prefix/suffix respectively. We also define
	\[
	K_{n}:=\left\{  u\in L_{n}(x):|O_{w}(u)|=0\right\}
	\]
	to be the set of $n$-letter words in $L(X)$ not containing $w$. Then
	partitioning words in $L_{n}(X)\setminus K_{n}$ by the first and last
	appearance of $w$, recalling that $w$ is synchronizing, gives the formula
	\[
	\left\vert L_{n}(X)\right\vert =\left\vert K_{n}\right\vert + \sum_{0\leq i <
		j\leq n}|S_{i}||R_{j-i} ||P_{n-j}|,
	\]

	thus
	\begin{equation}
	\label{sumproduct}\frac{\left\vert L_{n}(X)\right\vert }{\lambda^{n}} =
	\frac{\left\vert K_{n}\right\vert }{\lambda^{n}} + \sum_{0\leq i < j \leq n}
	\frac{|S_{i}|}{\lambda^{i}} \frac{|R_{j-i}|} {\lambda^{j-i}} \frac{|P_{n-j}%
		|}{\lambda^{n-j}}.
	\end{equation}

	We now wish to take the limit as $n \rightarrow\infty$ of both sides of
	(\ref{sumproduct}). First, we note that since $X$ is entropy minimal,
	$h_{top}(X_{w}) < h_{top}(X)$, where $X_{w}$ is the subshift of points of $X$
	not containing $w$. Therefore,
	\[
	\limsup_{n\rightarrow\infty}\frac{1}{n}\log\left\vert K_{n}\right\vert <
	h_{top}(X).
	\]
	Since all words in $P_{n}$ and $S_{n}$ are the concatenation of $w$ with a
	word in $K_{n}$, $|P_{n}|, |S_{n}| \leq|K_{n}|$, and so
	\[
	\limsup_{n\rightarrow\infty}\frac{1}{n}\log\left\vert P_{n}\right\vert
	,\limsup_{n\rightarrow\infty}\frac{1}{n}\log\left\vert S_{n}\right\vert <
	h_{top}(X),
	\]
	implying that the infinite series
	\[
	\sum_{n=0}^{\infty}\frac{\left\vert P_{n}\right\vert }{\lambda^{n}} \text{ and
	} \sum_{n=0}^{\infty}\frac{\left\vert S_{n}\right\vert }{\lambda^{n}}\text{
		converge.}%
	\]
	We now take the limit of the right-hand side of (\ref{sumproduct}).
	\[
	\lim_{n \rightarrow\infty} \frac{\left\vert K_{n}\right\vert }{\lambda^{n}} +
	\sum_{0\leq i < j \leq n} \frac{|S_{i}|} {\lambda^{i}} \frac{|R_{j-i}%
		|}{\lambda^{j-i}} \frac{|P_{n-j}|}{\lambda^{n-j}} = \lim_{n \rightarrow\infty}
	\sum_{0 \leq k \leq n} \left(  \frac{|R_{k}|}{\lambda^{k}} \left(  \sum_{i =
		0}^{n-k} \frac{|S_{i}|}{\lambda^{i}} \frac{|P_{n - k - i}|} {\lambda^{n-k-i}%
	}\right)  \right)  .
	\]

	Since $\frac{|R_{k}|}{\lambda^{k}}$ converges to the limit $\frac{\mu
		(w)}{\lambda^{\left\vert w\right\vert }}$ and the series
		$\sum_{m = 0}^{\infty} \sum_{i = 0}^{m} \frac{|S_{i}|}{\lambda^{i}} \frac{|P_{m-i}
		|}{\lambda^{n-k-i}}$ converges, the above can be rewritten as
	\begin{multline*}
	\lim_{n \rightarrow\infty} \sum_{0 \leq k \leq n} \left(  \frac{|R_{k}
		|}{\lambda^{k}} \left(  \sum_{i = 0}^{n-k} \frac{|S_{i}|}{\lambda^{i}}
	\frac{|P_{n - k - i}|}{\lambda^{n-k-i}}\right)  \right)  = \frac{\mu
		(w)}{\lambda^{\left\vert w\right\vert }} \lim_{m \rightarrow\infty} \sum_{m =
		0}^{\infty} \sum_{i = 0}^{m} \frac{|S_{i}|}{\lambda^{i}} \frac{|P_{m- i}
		|}{\lambda^{n-k-i}}\\
	= \frac{\mu(w)}{\lambda^{\left\vert w\right\vert }} \sum_{n=0}^{\infty}
	\frac{\left\vert P_{n}\right\vert }{\lambda^{n}} \sum_{n=0}^{\infty}
	\frac{\left\vert S_{n}\right\vert }{\lambda^{n}}.
	\end{multline*}

	Recalling (\ref{sumproduct}), we see that $\lim_{n \rightarrow\infty}
	\frac{|L_{n}(X)|}{\lambda^{n}}$ converges to this limit as well, completing
	the proof.
\end{proof}

We will be able to say even more about a class of synchronized subshifts
called the $S$-gap subshifts.

\begin{definition}\label{Sgap}
	Let $S\subseteq\mathbb{N} \cup \{0\}$. We define the $S-$gap subshift $X_{S}$ by the set
	of forbidden words $\{10^{n}1 \ : \ n \notin S\}$. Alternately, $X_{S}$ is the
	set of bi-infinite $\{0,1\}$ sequences where the gap between any two nearest
	$1$s has length in $S.$
\end{definition}

It is immediate from the definition that $1$ is a synchronizing word for every
$S-$gap subshift. Also, all $S$-gap subshifts are entropy minimal (see Theorem
C, Remark 2.4 of \cite{CT2}), and as long as $\gcd(S+1)=1$, their unique
measure of maximal entropy is mixing (in fact Bernoulli) by Theorem 1.6 of
\cite{Cl2}. (This theorem guarantees that the unique MME is Bernoulli up to
period $d$ given by the gcd of periodic orbit lengths, and it's clear that
$S+1$ is contained in the set of periodic orbit lengths.)

In this case Climenhaga \cite{Cl} conjectured that the limit $\lim
_{n\rightarrow\infty}\frac{\left\vert L_{n}(X_{S})\right\vert }{e^{nh_{top}%
		(X_{S})}}$ existed; we prove this and we give an explicit formula for the limit.


\begin{corollary}
	\label{limit2}Let $S\subseteq\mathbb{N}$ satisfy $\gcd(S+1)=1$, let
	$\mu$ be the unique MME on $X_{S}$, and let $\lambda = e^{h_{top}(X_{S})}$. Then 
	$\displaystyle\lim_{n\rightarrow\infty}\frac{\left\vert L_{n}(X_{S})\right\vert
	}{\lambda^n}$ exists and is equal to
	$\displaystyle \frac{\mu(1)\lambda}{(\lambda-1)^{2}}$ when $S$ is infinite and
	$\displaystyle \frac{\mu(1)\lambda (1 - \lambda^{-(\max S) - 1})^2}{(\lambda-1)^{2}}$ when $S$ is finite. 
\end{corollary}

\begin{proof}
	Using the notation of the proof of Theorem~\ref{limit}, we define $w=1$ and write
	$\lambda=e^{h_{top}(X_{S})}$. If $S$ is infinite, it is easy to see that $\left\vert
	P_{i}\right\vert =\left\vert S_{i}\right\vert = 1$ for all $i$. As noted above,
	$X_{S}$ is entropy minimal and its unique measure of maximal entropy is
	mixing, and so the proof of Theorem~\ref{limit} implies that
	\[
	\lim_{n\rightarrow\infty}\frac{\left\vert L_{n}(X_{S})\right\vert
	}{e^{nh_{top}(X_{S})}}=\frac{\mu(1)}{\lambda}\left(  \sum_{i=0}^{\infty}
	\frac{1}{\lambda^{i}}\right)  ^{2}=\frac{\mu(1)}{\lambda}\left(  \frac
	{1}{1-\lambda^{-1}}\right)  ^{2}=\frac{\mu(1)\lambda}{(\lambda-1)^{2}}.
	\]
	
	If instead $S$ is finite (say $M = \max S$), then the reader may check that 
	$\left\vert P_{i}\right\vert$ and $\left\vert S_{i}\right\vert$ are both equal to $1$ 
	for all $i \leq M$ and equal to $0$ for all $i > M$. Then, the proof of Theorem~\ref{limit} implies that
	\[
	\lim_{n\rightarrow\infty}\frac{\left\vert L_{n}(X_{S})\right\vert
	}{e^{nh_{top}(X_{S})}}=\frac{\mu(1)}{\lambda}\left(  \sum_{i=0}^{M}
	\frac{1}{\lambda^{i}}\right)  ^{2}=\frac{\mu(1)}{\lambda}\left(  \frac
	{1 - \lambda^{-M-1}}{1-\lambda^{-1}}\right)  ^{2} = \frac{\mu(1)\lambda(1 - \lambda^{-M-1})^2}{(\lambda - 1)^2},
	\]
	completing the proof.
	
\end{proof}

\bigskip As noted in \cite{Cl}, a motivation for proving the existence of
this limit is to fill a gap from \cite{spandl} for a folklore formula for the 
topological entropy of $X_{S}$. Two proofs of this formula are presented
in \cite{Cl}, and Corollary \ref{maincor} yields yet another proof. 


\begin{corollary}
	Let $S\subseteq\mathbb{N} \cup \{0\}$ with $\gcd(S+1)=1$. Then $h_{top}(X_{S})=\log\lambda$,
	where $\lambda$ is the unique solution of
	\[
	1=\sum_{n\in S}\lambda^{-n-1}.%
	\]
	
\end{corollary}

\begin{proof}
	For any $S$-gap shift $X_S$, we can write
\[
\left[1\right] = \left(\bigsqcup_{n = 0}^{\infty} \left[  10^{n}1\right]\right) \cup \{x \in X_S \ : \ x_0 = 1 \textrm{ and } \forall n > 0, x_n = 0\}.
\]
By shift-invariance, $\mu(10^{\infty})=0$, and so by Lemma~\ref{synchlem} and Corollary \ref{maincor},
	\[
	\mu(1)=\sum_{n\in S}\mu(10^{n}1)=\sum_{n\in S}\mu(1)e^{h_{top}(X_{S}%
		)(-n-1)}\text{.}%
	\]
Dividing both sides by $\mu(1)$ completes the proof.
	
\end{proof}

We also prove that for every $S-$gap subshift, the unique measure of maximal
entropy has highly constrained values, which are very similar to those of the
Parry measure for shifts of finite type. 

\begin{theorem}
	\label{value}Let $X_{S}$ be an $S-$gap subshift and $\mu$ the measure of
	maximal entropy. Then $\mu(1)=\frac{1}{\sum_{n\in S}(n+1)e^{-h_{top}(X_{S})(n+1)}}$,
	and for every $w\in L(X_{S})$, there exists a polynomial $f_{w}$ with integer
	coefficients so that $\mu(w)=k_{w}+\mu(1)f_{w}(e^{-h_{top}(X_{S})})$ for some
	integer $k_{w}$.
	
\end{theorem}

\begin{proof}
	As noted above, $S$-gap shifts are synchronized and entropy minimal, and so
	have unique measures of maximal entropy.
	
	Denote by $\mu$ the unique measure of maximal entropy for some $S-$gap
	subshift $X_{S}$, and for readability we define
	\[
	t=e^{-h_{top}(X)}.
	\]
	
	Since $X_S$ is entropy minimal, $\mu(1) > 0$, and so by the pointwise ergodic theorem
	(applied to $\chi_{[1]}$), $\mu$-a.e. point of $X_S$ contains infinitely many $1$s. 
	Therefore, we can partition points of $X_S$ according to the closest
	$1$ symbols to the left and right of the origin, and represent $X_S$ (up to a null set)
	as the disjoint union
	$\bigcup_{n\in S} \bigcup_{i=0}^{n} \sigma_i \left[10^{n}1\right]$. Then by 
	Lemma~\ref{synchlem} and Corollary \ref{maincor},
	\begin{align*}
	1  & =\sum_{n\in S}(n+1)\mu(10^{n}1)\\
	& =\sum_{n\in S}(n+1)\mu(1)t^{n+1}\text{,}%
	\end{align*}
	yielding the claimed formula for $\mu(1)$.
	
	Now we prove the general formula for $\mu(w)$, and will proceed by induction on the length $n$ of $w$. 
	For the base case $n=1$, $\mu(0) = 1 - \mu(1)$, verifying the theorem.

	Now, assume that the theorem holds for every $n\leq N$ for some $N \geq1$. Let
	$w\in L_{N-1}(X_{S})$, and we will verify the theorem for $1w1$, $1w0$, $0w1$,
	and $0w0$. If $1w1 \notin L(X_{S})$, then
	\begin{align*}
	\mu(1w1)  &  =0,\\
	\mu(1w0)  &  =\mu(1w) - \mu(1w1) = \mu(1w),\\
	\mu(0w1)  &  =\mu(w1) - \mu(1w1) = \mu(w1), \text{ and}\\
	\mu(0w0)  &  =1-\mu(1w1) - \mu(1w0) - \mu(0w1)= 1 - \mu(1w) - \mu(w1).
	\end{align*}
	The theorem now holds by the inductive hypothesis.
	
	If $1w1 \in L(X_{S})$, then as before $E_{X_{S}}(1w1)=E_{X_{S}}(1)$, implying
	\begin{align*}
	\mu(1w1)  &  = \mu(1) t^{1 + |w|},\\
	\mu(1w0)  &  = \mu(1w) - \mu(1w1) = \mu(1w) - \mu(1) t^{1 + |w|},\\
	\mu(0w1)  &  = \mu(w1) - \mu(1w1) = \mu(w1) - \mu(1) t^{1 + |w|}, \text{
		and}\\
	\mu(0w0)  &  = 1 - \mu(1w1) - \mu(1w0) - \mu(0w1) = 1 - \mu(1w) - \mu(w1) +
	\mu(1) t^{1 + |w|},
	\end{align*}
	again implying the theorem by the inductive hypothesis and completing the proof.
	







\end{proof}


\section{$\mathbb{G}-$subshifts}

\label{gsec}

Throughout this section, $\mathbb{G}$ will denote a countable amenable group
generated by a finite set $G=\left\{  g_{1},...,g_{d}\right\}  $ which is
torsion-free, i.e. $g^{n}=e$ if and only if $n=0$.
\subsection{Main result}

 For any $N=(N_{1}%
,...,N_{d})\in\mathbb{Z}_{+}^{d}$, we define $\mathbb{G}_{N}$ to be the
subgroup generated by $\left\{  g_{1}^{N_{1}},...,g_{d}^{N_{d}}\right\}  ,$
and use $\faktor{\mathbb{G}}{\mathbb{G}_{N}}$ to represent the collection 
$\left\{g\cdot\mathbb{G}_{N}:g\in\mathbb{G}\right\}$ of left cosets of $\mathbb{G}_N$. 
Clearly, $\left\vert
\faktor{\mathbb{G}}{\mathbb{G}_{N}}\right\vert =N_{1}N_{2}\cdots N_{d}$.

We again must begin with some relevant facts and definitions. The following
structural lemma is elementary, and we leave the proof to the reader.

\begin{lemma}
\label{one} For any amenable $\mathbb{G}$ and $F\Subset\mathbb{G}$, there
exists $N=(N_{1},...,N_{d})\in\mathbb{Z}_{+}^{d}$ such that for every
nonidentity $g \in\mathbb{G}_{N}$, $g\cdot F\cap F=\varnothing.$ 


\end{lemma}

As in the $\mathbb{Z}$ case, if $v,w\in L_{F}(\mathcal{A}^{\mathbb{G}})$ for
some $F\Subset\mathbb{G}$, we define the function $O_{v}:L(\mathcal{A}^{%
\mathbb{G}})\rightarrow \mathcal{P}(\mathbb{G})$ which sends a word to the set of
locations where $v$ appears as a subword, i.e. 
\begin{equation*}
O_{v}(u):=\left\{ g\in\mathbb{G}:\sigma_{g}(u) \in [v]\right\} . 
\end{equation*}
We also define the function $R_{u}^{v\rightarrow w}:O_{v}(u)\rightarrow L(%
\mathcal{A}^{\mathbb{G}})$, where $R_{u}^{v\rightarrow w}(g)$ is the word
you obtain by replacing the occurrence of $v$ at $g \cdot F$ within $u$ by $w
$.

We now again must define a way to replace many occurrences of $v$ by $w$
within a word $u$, but will do this via restricting the sets of locations
where the replacements occur rather than the pairs $(v,w)$. $\ $We say $%
S\subset\mathbb{G}$ is $F-$\textbf{sparse }if $g\cdot F\cap g^{\prime}\cdot
F=\varnothing$ for every unequal pair $g,g^{\prime}\in S$. When $v,w \in
L_{F}(X)$ and $S$ is $F-$sparse, we may simultaneously replace occurrences
of $v$ by $w$ at locations $g \cdot F$, $g \in S$ by $w$ without any of the
complications dealt with in the one-dimensional case, and we denote the
resulting word by $R_{u}^{v\rightarrow w}(S)$. Formally, $%
R_{u}^{v\rightarrow w}(S)$ is just the image of $u$ under the composition of 
$R_{u}^{v \rightarrow w}(s)$ over all $s \in S$.


The following lemmas are much simpler versions of Lemmas \ref{injective} and 
\ref{preimage} for $F$-sparse sets.

\begin{lemma}
\label{Ginjective} For any $F$, $v,w\in L_{F}(X)$, and $F$-sparse set $T
\subseteq O_{v}(u)$, $R_{u}^{v\rightarrow w}$ is injective on subsets of $T$.
\end{lemma}

\begin{proof}
Fix $F, u, v, w, T$ as in the lemma. If $S \neq S' \subseteq T$, then
either $S \setminus S'$ or $S' \setminus S$ is nonempty; assume without 
loss of generality that it is the former. Then, if $s \in S \setminus S'$,
by definition $(R_{u}^{v\rightarrow w}(S))_{s + F} = w$ and
$(R_{u}^{v\rightarrow w}(S'))_{s + F} = v$, and so
$R_{u}^{v\rightarrow w}(S) \neq R_{u}^{v\rightarrow w}(S')$.
\end{proof}

\begin{lemma}
\label{Gpreimage} For any $F$ and $v,w\in L_{F}(X)$, any $F$-sparse set $T
\subseteq O_{v}(u)$, any $u^{\prime}$, and any $m \leq|T \cap
O_{w}(u^{\prime })|$, 
\begin{equation*}
|\{(u, S) \ : \ S \text{ is $F$-sparse}, |S| = m, S \subseteq T, u^{\prime}=
R_{u}^{v\rightarrow w}(S)\}| \leq{\binom{|T \cap O_{w}(u^{\prime})|}{m}}. 
\end{equation*}
\end{lemma}

\begin{proof}
Fix any such $F, u', v, w, T, m$ as in the lemma. Clearly, for any $S$, 
$S \subseteq O_w(R_{u}^{v\rightarrow w}(S))$, and so if $R_{u}^{v\rightarrow w}(S) = u'$,
then $S \subseteq O_w(u')$. There are only
${\binom{|T \cap O_{w}(u^{\prime})|}{m}}$ choices for $S \subseteq T \cap O_w(u')$ with $|S| = m$, and 
an identical argument to that of Lemma~\ref{preimage} shows that for each such $S$, 
there is only one $u$ for which $R_{u}^{v\rightarrow w}(S) = u'$.
\end{proof}


Whenever $v,w\in L_{F}(X)$ and $E_{X}(v)\subseteq E_{X}(w)$, clearly $%
R_{u}^{v\rightarrow w}(S)\in L(X)$ for any $F$-sparse set $S\subseteq
O_{v}(u)$; this, along with the use of Lemma \ref{one}, will be the keys to the
counting arguments used to prove our main result for $\mathbb{G}$-subshifts. 

\begin{theorem}
\label{Gtheorem} Let $X$ be a $\mathbb{G-}$subshift, $\mu$ a measure of
maximal entropy of $X$, $F\Subset\mathbb{G}$, and $v,w\in L_{F}(X).$ If $%
E(v)\subseteq E(w)$ then 
\begin{equation*}
\mu(v)\leq\mu(w). 
\end{equation*}
\end{theorem}

\begin{proof}
Take $\mathbb{G}$, $X$, $\mu$, $F$, $v$, and $w$ as in the theorem, and
suppose for a contradiction that $\mu(v) > \mu(w)$. Choose any $\delta \in%
\mathbb{Q}_{+}$ with $\delta< \frac{\mu(v) - \mu(w)}{5}$.
Let $F_n$ be a Følner sequence satisfying Theorem~\ref{SMBthm}.
For every $n\in\mathbb{Z}_{+},$ we define 
\begin{equation*}
S_{n}:=\left\{ u\in L_{F_{n}}(X):\left\vert O_{v}(u)\right\vert
\geq\left\vert F_{n}\right\vert (\mu(v)-\delta)\text{ and }\left\vert
O_{w}(u)\right\vert \leq\left\vert F_{n}\right\vert (\mu(w)+\delta)\right\}
. 
\end{equation*}

By the pointwise ergodic theorem (applied to $\chi_{[v]}$ and $\chi_{[w]}$), 
$\mu(S_{n})\rightarrow1$, and then by Corollary~\ref{SMBcor}, 
\begin{equation}
\lim_{n\rightarrow\infty}\frac{\log|S_{n}|}{n}=h_{top}(X).   \label{Snbound}
\end{equation}

Let $N\in\mathbb{Z}_{+}^{d}$ be a number obtained by Lemma \ref{one} that is
minimal in the sense that if any of the coordinates is decreased then it
will not satisfy the property of the lemma.

We note that for every $u\in S_{n}$, $|O_{v}(u)|-|O_{w}(u)|>3\delta|F_{n}|$.
Therefore, for every $u\in S_{n}$, there exists $h(u)\in \faktor{%
\mathbb{G}}{\mathbb{G}_{N}}$ such that 
\begin{equation}
\left\vert O_{v}(u)\cap h(u) \right\vert -\left\vert O_{w}(u)\cap h(u)
\right\vert >\frac{3\delta}{M} |F_{n}|\text{,}   \label{ineq}
\end{equation}
where $M=\left| \faktor{\mathbb{G}}{\mathbb{G}_{N}}\right| .$

For every $u\in S_{n}$, define $k_{n}(u)\in\mathbb{N}$ satisfying $\left\vert
O_{v}(u)\cap h(u)\right\vert \in [ k_{n} (u)|F_{n}|\frac{\delta}{M},$ 
\newline $(k_{n}(u)+1)|F_{n}|\frac{\delta}{M}] $.

Using $M=\left| \faktor{\mathbb{G}}{\mathbb{G}_{N}}\right| $ and the fact
that $3\leq k_{n}(u)\leq\frac{M}{\delta}$, we may choose $S_{n}^{\prime
}\subseteq S_{n}$ with $|S_{n}^{\prime}|\geq\frac{|S_{n}|}{M^{2}/\delta}$, $%
h_{n} \in\faktor{\mathbb{G}}{\mathbb{G}_{N}}$ and $k_{n}\in\mathbb{N}$ such
that for every $u\in S_{n}^{\prime}$ we have $h(u)=h_{n}$ and $k_{n}(u)=k_{n}
$. This implies that for every $u\in S_{n}^{\prime}$ 
\begin{align*}
\left\vert O_{v}(u)\cap h_{n}(u)\right\vert & \geq(k_{n}+1)|F_{n}|\frac{%
\delta}{M}\text{, and hence} \\
\left\vert O_{w}(u)\cap h_{n}(u)\right\vert & \leq(k_{n}-2)|F_{n}|\frac{%
\delta}{M}\text{ (using (\ref{ineq})).}
\end{align*}

By the pigeonhole principle, we may pass to a sequence on which $h_{n} = h$
and $k_{n} = k$ are constant, and for the rest of the proof consider only $n$
in this sequence. Let $\varepsilon\in\mathbb{Q}_{+}$with $\varepsilon <\frac{%
\delta}{|F \cdot F^{-1}|}$. For each $u\in S_{n}^{\prime}$, we define 
\begin{equation*}
A_{u}:=\left\{ R_{u}^{v\rightarrow w}(S):S\subseteq O_{v}(u)\cap h \text{ and }%
\left\vert S\right\vert =\varepsilon\left\vert F_{n}\right\vert /M\right\} 
\end{equation*}
(without loss of generality we may assume $\varepsilon\left\vert
F_{n}\right\vert /M$ is an integer by taking a sufficiently large $n$) $.$

Since $E_{X}(v)\subseteq E_{X}(w)$, we have that $A_{u}\subset L(X).$ By
Lemma~\ref{Ginjective}, 
\begin{equation*}
|A_{u}| \geq{\binom{|O_{v}(u)\cap h|} {\varepsilon\left\vert
F_{n}\right\vert /M}} \geq{\binom{\delta k |F_{n}|/M} {\varepsilon\left\vert
F_{n}\right\vert /M}}. 
\end{equation*}

On the other hand, for every $u^{\prime}\in\bigcup_{u\in S_{n}}A_{u}$, we
have that 
\begin{equation*}
\left\vert O_{w}(u^{\prime})\cap h\right\vert \leq\frac{|F_{n}|}{M}\left(
(k_{n}-2)\delta+\varepsilon|F \cdot F^{-1}|\right) \leq\frac{\delta|F_{n}|}{M}%
(k_{n}-1). 
\end{equation*}
(here, we use $\left\vert O_{w}(u)\cap h(u)\right\vert \leq(k_{n}-2)|F_{n}|%
\frac{\delta}{M}$ plus $\left\vert S\right\vert =\varepsilon \left\vert
F_{n}\right\vert /M$ and the simple fact that a replacement of $v$ by $w$ in 
$u$ can create at most $|F \cdot F^{-1}|$ new occurrences of $w$.) Therefore, by
Lemma~\ref{Gpreimage}, 
\begin{equation*}
\left\vert \left\{ u\in S_{n}^{\prime}:u^{\prime}\in A_{u}\right\}
\right\vert \leq{\binom{\delta(k_{n}-1)|F_{n}|/M}{\varepsilon\left\vert
F_{n}\right\vert /M}.} 
\end{equation*}

By combining the two inequalities, we see that 
\begin{equation}
|L_{n}(X)|\geq\left\vert \bigcup_{u\in S_{n}^{\prime}}A_{u}\right\vert
\geq|S_{n}^{\prime}|{\binom{\delta k_{n}|F_{n}|/M}{\varepsilon\left\vert
F_{n}\right\vert /M}}{\binom{\delta(k_{n}-1)|F_{n}|/M}{\varepsilon\left\vert
F_{n}\right\vert /M}}^{-1}.
\end{equation}
Now, we take logarithms of both sides, divide by $\left\vert
F_{n}\right\vert $, and let $n$ approach infinity (along the earlier defined
sequence). Then we use the definition of entropy, the inequality $%
|S_{n}^{\prime}|\geq \frac{|S_{n}|}{M^{2}/\delta}$, (\ref{Snbound}), and
Stirling's approximation to yield 
\begin{align*}
h_{top}(X) & \geq h_{top}(X)+\frac{\varepsilon}{M}\bigg[\left( \frac{\delta k%
}{\varepsilon}\log\frac{\delta k}{\varepsilon}-\left( \frac{\delta k}{%
\varepsilon}-1\right) \log\left( \frac{\delta k}{\varepsilon}-1\right)
\right) \\
& -\left( \frac{\delta(k-1)}{\varepsilon}\log\frac{\delta(k-1)}{\varepsilon }%
-\left( \frac{\delta(k-1)}{\varepsilon}-1\right) \log\left( \frac {%
\delta(k-1)}{\varepsilon}-1\right) \right) \bigg].
\end{align*}

Since the function $x \log x - (x - 1) \log(x-1)$ is strictly increasing for 
$x > 1$, the right-hand side of the above is strictly greater than $%
h_{top}(X)$, a contradiction. Therefore, our original assumption does not hold
and hence $\mu(v) \leq\mu(w)$.
\end{proof}

\subsection{Applications to hereditary subshifts}

\label{hered}

One class of $\mathbb{G-}$subshifts with many pairs of words satisfying
$E_{X}(v)\subsetneq E_{X}(w)$, allowing for the use of Theorem~\ref{Gtheorem}, are the hereditary subshifts 
(introduced in \cite{KL1}).

A partial order $\leq$ on a finite set $\mathcal{A}$ induces a partial order
on $\mathcal{A}^{n}$ and $\mathcal{A}^{\mathbb{G}}$ (coordinatewise) which
will also be denoted by $\leq$. When $\mathcal{A=}\left\{  0,1...,m\right\}  $
we will always use the linear order $0 \leq1 \leq\ldots\leq m$.

\begin{definition}
	Let $X\subseteq\mathcal{A}^{\mathbb{G}}$ be a subshift and $\leq$ a partial
	order on $\mathcal{A}$. We say $X$ is $\leq-$\textbf{hereditary (or simply
		hereditary)} if for every $x\in\mathcal{A}^{\mathbb{G}}$ such that there
	exists $y\in X$ such that $x\leq y$ then $x\in X.$
\end{definition}

Examples of hereditary shifts include $\beta-$shifts \cite{Kw},
$\mathscr{B}-$free shifts (\cite{KLW}), spacing shifts (\cite{LZ}),
multi-choice shifts (\cite{LMP}) and bounded density shifts (\cite{S}). Many
of these examples have a unique measure of maximal entropy, but not every
hereditary subshift has this property (see \cite{KLW})
.

This definition immediately implies that whenever $x \leq y$ for $x, y \in
L(X)$, $E_{X}(y) \subseteq E_{X}(x)$, yielding the following corollary of
Theorem~\ref{hardcase}.

\begin{corollary}
	Let $X$ be a $\leq-$hereditary $\mathbb{G-}$subshift, $\mu$ a measure of maximal entropy,
	and $v,w\in L_{n}(X)$ for some $n \in\mathbb{N}.$ If $u\leq v$ then
	$\mu(v)\leq\mu(u).$
\end{corollary}

In particular, if $\mathcal{A=}\left\{  0,1...,m\right\}$, then $\mu(m)\leq\mu(m-1)...\leq\mu(1)\leq\mu(0)$.

Having $u\leq v$ $\ $is sufficient but not necessary for $E(v)\subseteq E(w).
$ In particular, for $\beta-$shifts and bounded density shifts, there are many
other pairs (with different lengths) where this happens. This is due to an
additional property satisfied by these hereditary shifts.

\begin{definition}
	Let $X\subseteq\left\{  0,1,...,m\right\}  ^{\mathbb{Z}}$ be a hereditary
	$\mathbb{Z}$-subshift. We say $X$ is \textbf{$i$-hereditary} if for every $u\in L_{n}(X)$
	and $u^{\prime}$ obtained by inserting a $0$ somewhere in $u$, it is the case
	that $u^{\prime}\in L_{n+1}(X)$.
\end{definition}

In particular, $\beta-$shifts and bounded density shifts are $i$-hereditary,
but not every spacing shift is $i$-hereditary. It's immediate that any
$i$-hereditary shift satisfies $E_{X}(0^{j}) \subseteq E_{X}(0^{k})$ whenever
$j \geq k$. We can get equality if we assume the additional property of specification.

\begin{definition}
	A $\mathbb{Z}$-subshift $X$ has the \textbf{specification property (at distance }%
	$N$)\textbf{\ }if for every $u,w\in L(X)$ there exists $v\in L_{N}(X)$ such
	that $uvw\in L(X).$
\end{definition}

Clearly, if $X$ is hereditary and has specification property at distance $N$,
then $u0^{N}w$ and $u0^{N+1}w\in L(X)$ for all $u,w\in L(X)$, and so in this
case $E_{X}(0^{N})=E_{X}(0^{N+1})$. We then have the following corollary of
Theorem~\ref{hardcase}.

\begin{corollary}
	\label{hereditary}Let $X\subseteq\left\{  0,1,...,m\right\}  ^{\mathbb{Z}}$ be
	a i-hereditary $\mathbb{Z-}$subshift$.$ Then for every $n\in\mathbb{Z}_{+}$
	\[
	h_{top}(X)\geq\log\frac{\mu(0^{n})}{\mu(0^{n+1})}.
	\]

	Furthermore, if $X$ has the specification property at distance $N$, then
	\[
	h_{top}(X)=\log\frac{\mu(0^{N})}{\mu(0^{N+1})}.
	\]
	
\end{corollary}

We note that if $X$ has the specification property at distance $N$, then it
also has it at any larger distance. Therefore, the final formula can be
rewritten as
\begin{multline*}
h_{top}(X)=\lim_{N\rightarrow\infty}\log\frac{\mu(0^{N})}{\mu(0^{N+1})}%
=\lim_{N\rightarrow\infty}-\log\mu(x(0)=0\ |\ x_{[-N,-1]}=0^{N})\\
=-\log\mu(x(0)=0\ |\ x_{(-\infty,-1]}=0^{\infty}),
\end{multline*}
recovering a formula (in fact a more general one for topological pressure of
$\mathbb{Z}^{d}$ SFTs) proved under different hypotheses in \cite{MP}.

\end{document}